\documentclass[10pt, a4paper]{article}
\usepackage{graphicx,latexsym, amsmath,amsfonts}
\usepackage{amsthm}
\usepackage{enumerate}
\usepackage{float} 

\makeatletter \@addtoreset{equation}{section}

\begin{document}

\linespread{1.3}

\newcommand{\E}{\mathbb{E}}
\newcommand{\PP}{\mathbb{P}}
\newcommand{\RR}{\mathbb{R}}
\newcommand{\NN}{\mathbb{N}}

\newtheorem{theorem}{Theorem}[section]
\newtheorem{lemma}[theorem]{Lemma}
\newtheorem{coro}[theorem]{Corollary}
\newtheorem{defn}[theorem]{Definition}
\newtheorem{assp}[theorem]{Assumption}
\newtheorem{expl}[theorem]{Example}
\newtheorem{prop}[theorem]{Proposition}
\newtheorem{rmk}[theorem]{Remark}
\newtheorem{notation}[theorem]{Notation}

\newcommand\tq{{\scriptstyle{3\over 4 }\scriptstyle}}
\newcommand\qua{{\scriptstyle{1\over 4 }\scriptstyle}}
\newcommand\hf{{\textstyle{1\over 2 }\displaystyle}}
\newcommand\hhf{{\scriptstyle{1\over 2 }\scriptstyle}}

\newcommand{\eproof}{\indent\vrule height6pt width4pt depth1pt\hfil\par\medbreak}

\def\a{\alpha} \def\g{\gamma}
\def\e{\varepsilon} \def\z{\zeta} \def\y{\eta} \def\o{\theta}
\def\vo{\vartheta} \def\k{\kappa} \def\l{\lambda} \def\m{\mu} \def\n{\nu}
\def\x{\xi}  \def\r{\rho} \def\s{\sigma}
\def\p{\phi} \def\f{\varphi}   \def\w{\omega}
\def\q{\surd} \def\i{\bot} \def\h{\forall} \def\j{\emptyset}

\def\be{\beta} \def\de{\delta} \def\up{\upsilon} \def\eq{\equiv}
\def\ve{\vee} \def\we{\wedge}

\def\t{\tau}

\def\F{{\cal F}}
\def\T{\tau} \def\G{\Gamma}  \def\D{\Delta} \def\O{\Theta} \def\L{\Lambda}
\def\X{\Xi} \def\S{\Sigma} \def\W{\Omega}
\def\M{\partial} \def\N{\nabla} \def\Ex{\exists} \def\K{\times}
\def\V{\bigvee} \def\U{\bigwedge}

\def\1{\oslash} \def\2{\oplus} \def\3{\otimes} \def\4{\ominus}
\def\5{\circ} \def\6{\odot} \def\7{\backslash} \def\8{\infty}
\def\9{\bigcap} \def\0{\bigcup} \def\+{\pm} \def\-{\mp}
\def\<{\langle} \def\>{\rangle}

\def\lev{\left\vert} \def\rev{\right\vert}
\def\1{\mathbf{1}}

\def\tl{\tilde}
\def\trace{\hbox{\rm trace}}
\def\diag{\hbox{\rm diag}}
\def\for{\quad\hbox{for }}
\def\refer{\hangindent=0.3in\hangafter=1}

\newcommand{\Xtk}{X_{t_{k}}}
\newcommand{\Xtkk}{X_{t_{k+1}}}

\newcommand\wD{\widehat{\D}}

\title
	{ \bf Convergence, Non-negativity and Stability of a New Milstein Scheme with Applications to Finance}

	\author{   Desmond J. Higham%
		\thanks{%
			  Department of Mathematics and Statistics,
			  University of Strathclyde,
			   Glasgow, G1 1XH, Scotland, UK
			  (\texttt{d.j.higham@strath.ac.uk}).
			 }
	\and
		Xuerong Mao%
		\thanks{%
			  Department of Mathematics and Statistics,
			  University of Strathclyde,
			   Glasgow, G1 1XH, Scotland, UK
			  (\texttt{x.mao@strath.ac.uk}).
			 }
	\and
		Lukasz Szpruch%
		  \thanks{%
			  Mathematical Institute,
			  University of Oxford,
			  24-29 St Giles,
			  Oxford OX1 3LB, UK
			  (\texttt{szpruch@maths.ox.ac.uk}).
			  }
		    }
	\date{}

	\maketitle




	\begin{abstract}
	\textsf{\em We propose and analyse a new 
	Milstein type scheme for simulating stochastic differential equations (SDEs) with highly nonlinear
	coefficients. Our work is motivated by the need to justify multi-level Monte Carlo simulations for
	mean-reverting financial models with polynomial growth in the diffusion term.
	We introduce a double implicit Milstein scheme and show that it possesses desirable
	properties. It converges strongly and preserves non-negativity for a rich family of financial 
	models and can reproduce linear and nonlinear stability behaviour of the underlying
	SDE without severe restriction on the time step. Although the scheme is implicit, we point
	out examples of financial models where an explicit formula for the solution to the scheme
	can be found. 
	}

	\medskip

	\noindent \textsf{{\bf Key words: } \em Milstein scheme, implicit schemes 
	stochastic differential equation, stability, strong convergence, non-negativity}

	\medskip
	\noindent{\small\bf 2000 Mathematics Subject Classification: } 60H10,\;65J15

	\end{abstract}

	\pagenumbering{arabic}

	\section{Introduction} \label{sec:intro}

	We study numerical approximation of the scalar
	stochastic differential equation (SDE)
	\begin{equation}   \label{eq:SDE}
	dx(t)=f(x(t))dt+g(x(t))dw(t).
	\end{equation}
	Here $x(t)\in \RR$ for each $t\ge 0$,
	and,  for simplicity, 
	$x(0)$ is taken to be constant.
	 We  assume
	that $f\in C^{1}(\RR,\RR)$ and $g\in C^{2}(\RR,\RR)$.
	Throughout we let
	$(\Omega, {\mathcal{F}},\{{\mathcal{F}}_t\}_{t\geq 0}, \PP)$ be a
	complete probability space with a filtration
	$\{{\mathcal{F}}_t\}_{t\geq 0}$ satisfying the usual conditions,
	that is, right continuous and increasing while
	${\mathcal{F}}_0$ contains all $\PP$-null sets, and we let $w(t)$
	be a  Brownian motion defined on the probability space.

	Numerical approximations for equation \eqref{eq:SDE} are well
	studied in the case of globally Lipschitz continuous coefficients
	\cite{kloeden1992numerical}. Super-linearly
	growing coefficients, however, raise many new questions. 
	An important example is the Heston stochastic
	volatility 3/2-model \cite{heston1997simple,lewis2000option}:
	\begin{equation} \label{eq:heston}
	dx(t)=x(t)(\mu -\a x(t))dt+\be x(t)^{3/2}dw(t), \quad \mu,\a,\be>0.
	\end{equation}
	This equation is also known as an inverse square-root process 
	and was shown to be 
	useful for modelling term structure dynamics 
	\cite{ahn1999parametric}. We may also view  
	 \eqref{eq:heston} as a stochastic extension to the logistic equation \cite{gard1988introduction}.
	Recently \cite{hutzenthaler2011strong} demonstrated 
	that the standard 
	Euler-Maruyama (EM) discretization scheme can diverge in strong and weak senses
	for SDEs with super-linear coefficients.  
	However, in 
      \cite{szpruch-diss,Szpruch2010monotone} 
       it is shown that 
	an implicit Euler-type method   
	strongly converges for nonlinearities similar to \eqref{eq:heston}.
	These positive results rely heavily on a one-sided 
	Lipschitz condition satisfied by  
	the drift coefficients of the SDE \eqref{eq:SDE}, that is,
	for some constant $K$, 
	\begin{equation} \label{con:one-sided}
	 (x - y)(f(x) - f(y) ) \le K \lev x-y \rev^{2}, \quad \hbox{for } x,y \in\RR.
	\end{equation}
	The solution of \eqref{eq:heston} is non-negative and 
	condition \eqref{con:one-sided}
	holds in the relevant region $x,y\ge0$. 
	However, in general, numerical approximations do not preserve non-negativity 
	and hence convergence theorems developed in 
	\cite{szpruch-diss,Szpruch2010monotone} cannot 
	be applied in this situation. 

	Preservation of positivity 
     is a 
	desirable modeling property, and, in many cases,
        non-negativity 
	of the numerical approximation is needed in order for the scheme to be
	well defined. For example, evaluating the diffusion coefficient in \eqref{eq:heston}
	for a negative argument does not make sense. Many fixes have been proposed in
	the literature, but these can lead to substantial bias \cite{lord-comparison}.
For more information about positivity preservation we refer the reader to
\cite{Appleby2010,kahl2008structure,schurz2005convergence,szpruch-strongly}.
It was shown in \cite{kahl2006balanced} that a class of balanced methods can
preserve positivity under an appropriate choice of the weight functions, but
strong convergence was proved only under a 
global Lipschitz condition \cite{milstein1998balanced}.
Kahl et al.\ \cite{kahl2008structure} proved
that the classical EM scheme does not preserve positivity for
any scalar SDE. In the case of the drift implicit EM scheme positivity
can be preserved, \cite{szpruch-strongly}, if the drift coefficient has a very
special form, for example as in the Ait-Sahalia interest rate model \cite{ait1996testing}.
It was also shown in \cite{kahl2008structure} that the drift implicit
Milstein scheme applied to
\begin{equation} \label{eq:ex1}
dx(t)=\a(\mu-x(t))dt+\be x(t)^{p}dw(t) \quad  \a,\mu,\be>0, \mathrm{and }\, \, \, p\in[0.5],
\end{equation}
preserves positivity with no restriction on the time step if $p=0.5$. For $p\in(0.5,1]$ we need
to restrict the step size to lie below $\be^{-2}$.

Higher order approximation carries some pitfalls. 
It was demonstrated in \cite{higham2000stability} that
Milstein applied to a linear scalar SDE has worse stability properties
than EM, even once we allow for implicitness in the drift coefficient.
This is undesirable, particularly in the 
multi-level Monte Carlo 
(MLMC) setting 
\cite{giles2006improved,giles2008multilevel}, 
where we are required 
to use many simulations with large discretization time step.  
 It is therefore natural to look for more advanced numerical
techniques that automatically capture such a property. 

In order to address the issues mentioned above, we introduce a new
 $(\o,\s)$-Milstein scheme for a general scalar SDE.
Given any step size $\Delta
t$,  we define the partition $\mathcal{P}_{\Delta
t}:=\{t_{k}=k\Delta t :k=0,1,2,...\}$ of the half line
$[0,\infty)$. Letting $X_{t_{k}}$ denote the approximation to $x(t_{k})$,
with $X_{t_{0}}=x(0)$ and $\D w_{t_{k}}=w(t_{k+1})-w(t_{k})$,
the $(\o,\s)$-Milstein-scheme then has the following form
\begin{align} \label{eq:os-Milstein}
X_{t_{k+1}}& = X_{t_{k}}+\o f(X_{t_{k+1}})\D t
                 + (1-\o) f(X_{t_{k}})\D t +
                 g(X_{t_{k}})\D w_{t_{k}}
                  + \frac{1}{2}L^{1}g(X_{t_{k}})\D w_{t_{k}}^{2} \nonumber\\
            & - \frac{(1-\s)}{2}L^{1}g(X_{t_{k}}) \D t
                - \frac{\s}{2}L^{1}g(X_{t_{k+1}}) \D t,
 \end{align}
where $0 \le \o,\s \le 1$ are free parameters and $L^{1} = g\frac{\partial}{\partial x}$.
We note that the $(0,0)$-Milstein scheme reduces to classical
 Milstein \cite{milstein2004stochastic}. We will sometimes refer to $(1,1)$-Milstein
 as the double implicit scheme.
The advantage of the extra degree of implicitness offered by $\sigma$ will become clear as we
analyse the method.  
We note that the 
$(\o,\s)$-Milstein scheme can be naturally derived from 
the It\^{o}-Stratonovich expansion. Indeed, we can rewrite SDE \eqref{eq:SDE} 
into its Stratonovich form 
\begin{equation*}  
dx(t)=\underline{f}(x(t))dt+g(x(t))\circ dw(t),
\end{equation*}
where $\underline{f}(x)=f(x) - L^{1}g(x)$.
In the scalar case the drift-implicit  Milstein
scheme for the Stratonovich SDE is given by (see \cite{kloeden1992numerical} p. 345) 
\begin{align*}
\bar{X}_{t_{k+1}}& = \bar{X}_{t_{k}} + \underline{f}(\bar{X}_{t_{k+1}})\D t 
            + g(\bar{X}_{t_{k}}) \D w_{t_{k}}
             + \frac{1}{2}L^{1}g(\bar{X}_{t_{k}})\D w_{t_{k}}^{2}.
\end{align*}
Hence, we note that $(\o,\s)$-Milstein may be obtained from  the implicit
Milstein scheme for a Stratonovich SDE.   

In this work, we allow for a nonlinear drift coefficient
 and show that once $p>0.5$ in \eqref{eq:ex1} the step size restriction for non-negativity 
can be eliminated by the $(\o, \s)$-Milstein method. 
We also present fairly
general conditions for a family of Milstein schemes
to preserve positivity. Due to that property the one-sided Lipschitz 
structure 
(\ref{con:one-sided})
will not be lost. Hence, the new scheme can be shown to 
converge strongly to the solution of the SDE \eqref{eq:heston}.
Numerically we observe that the rate of strong convergence is $1$,
which 
Giles 
\cite{giles2006improved,giles2008multilevel}
has shown to be the optimal rate 
from the MLMC perspective. 

The material is structured as follows.
Section 2 contains proofs of the existence of positive solutions
to \eqref{eq:os-Milstein}.  In Section 3 we consider stability properties of
the double implicit scheme. As motivation we demonstrate that for linear test SDEs
we can significantly improve stability properties of the Milstein scheme. We then
extend this result to a more general nonlinear setting. In Section 4 
we develop the
convergence results. We give conclusions in Section 5. 
%
%

\section{Existence of a Solution for the Implicit Schemes} \label{sec:solution}
We begin with conditions that guarantee the existence of 
a unique  
solution to \eqref{eq:os-Milstein}.
These will motivate 
the assumptions that we use to force positivity.
\begin{lemma} \label{Zeidler}
Let $F$ be a function defined on $\RR$ and consider the equation
\begin{equation} \label{eq:F}
F(x)=b,
\end{equation}
for a given $b \in \RR$. If $F$ is strictly monotone, i.e.,
\begin{equation} \label{eq:strict mono}
(x-y)(F(x)-F(y))>0,
\end{equation}
for all $x,y\in\RR $, $x\neq y$, then equation (\ref{eq:F}) has
at most one solution. If $F$ is continuous and coercive, i.e.,
\begin{equation}
\lim_{\lev x \rev\rightarrow \8}\frac{xF(x)}{\lev x \rev}=\8,
\end{equation}
then for every $b \in\RR$, equation (\ref{eq:F}) has a solution $x
\in \RR$. Moreover, the inverse operator $F^{-1}$ exists.
\end{lemma}
A proof follows directly from Theorem 26.A in \cite{Zeidler1985}.
In order to prove that the $(\o,\s)$-Milstein (\ref{eq:os-Milstein}) scheme
is well defined we impose two conditions.
\begin{assp} \label{as:one}
Coefficients $f$ and $g$ in \eqref{eq:SDE}  are locally Lipschitz continuous and satisfy the
following two conditions: \newline
\textit{\underline{One-sided Lipschitz condition.}} There exists a constant $K>0$ such that
\begin{equation} \label{as:mono1}
(x-y)(f(x)-f(y))\leq K\lev x-y \rev^{2}  \quad \hbox{for all} \quad x,y\in \RR.
\end{equation}
\textit{\underline{Monotone condition.}} Operator $L^{1}$ acting on $g$ satisfies 
\begin{equation} \label{as:diff1}
(x-y)(L^{1}g(x)-L^{1}g(y) )\ge 0 \quad \hbox{for all} \quad  x,y\in\RR.
\end{equation}
\end{assp}
\begin{rmk} \label{rmk}
From Assumption \ref{as:one} and the Young inequality we may 
show that
the drift coefficient $f$ satisfies a one-sided Lipschitz-type condition
\[
 x f(x) \le K \lev x^{2}\rev + xf(0 ) 
\le a + b\lev x\rev^{2} \quad \hbox{for all} \quad  x\in\RR,
\]
where $a=0.5 \lev f(0)\rev^{2}$ and $b=(2K+1)/2$.
Also from Assumption \ref{as:one} we can show that $x  L^{1}g(x)$ 
is bounded   
below by a linear function
\begin{equation} \label{eq:lowerbound}
 x  L^{1}g(x) \ge  x L^{1}g(0) \quad \hbox{for all} \quad  x,y\in\RR.
\end{equation}
\end{rmk}
\begin{lemma} \label{defbem}
Define, for any given $\D t < ( \o (K+1))^{-1}$,
\begin{equation}
F(x)=x-\o f(x) \D t+\frac{\s}{2}L^{1}g(x), \quad x\in \RR.
\label{eq:Fdef}
\end{equation}
Then under Assumption \ref{as:one}, for any $b \in \RR$ 
there exists a unique $x \in \RR$ such that 
$F(x)=b$ and hence the method
\eqref{eq:os-Milstein} is well defined.
\end{lemma}
\begin{proof} In view of Lemma \ref{Zeidler} it is enough to
show that the function $F$ is continuous, coercive and strictly
monotone. Clearly, $F(x)$ is continuous on $\RR$. By
Assumption \ref{as:one},
\begin{equation*}
(x-y)(F(x)-F(y))
\ge
\lev x-y \rev^{2}-\o K\Delta t \lev x-y\rev^{2}
=(1-\o K\Delta t)\lev x-y\rev^{2}
>0,
\end{equation*}
for $\D t < ( \o (K+1))^{-1}$. Also by Assumption \ref{as:one}
and Remark \ref{rmk} 
\begin{align} \label{eq:coerc}
 xF(x) & =  x(x-\o f(x)\D t +\frac{\s}{2}L^{1}g(x)\D t) \\
& \ge  \lev x\rev^{2} ( 1 - \o \frac{2K+1}{2} \D t )
- \frac{\o}{2}x \lev f(0) \rev^{2} \D t +\frac{\s}{2}  x L^{1}g(0)  \D t, \nonumber
\end{align}
so $F$ is coercive.  
\end{proof}
From now on we assume that  
\begin{equation} \label{eq:time_re}
 \D t < ( \o (K+1))^{-1}
\end{equation}
\subsection{Existence of a Positive Solution for
the $(\o,\s)$-Milstein Scheme}
In this subsection we introduce assumptions on coefficients
$f$ and $g$ of equation \eqref{eq:SDE} that allow us to prove the existence
of a positive solution to 
\eqref{eq:os-Milstein}.
\begin{defn} \label{def:pos}
Given $x(0)>0$, if the solution of \eqref{eq:SDE} satisfies
$ \PP(\{x(t) > 0:t>0 \})= 1$ $( \PP(\{x(t) \ge 0:t>0 \})= 1)$,
then a stochastic one-step integration scheme computing 
approximations $X_{t_{k}} \approx x(t_{k})$
preserves positivity (non-negativity) if
\begin{equation*}
\PP(\{X_{t_{k+1}} > 0|X_{t_{k}} > 0 \})=1\quad  (\PP(\{X_{t_{k+1}} \ge 0|X_{t_{k}} \ge 0 \})=1).
\end{equation*}
\end{defn}
Let us note that to use the ideas from Lemma \ref{defbem} to prove 
the existence of a positive 
solution to the implicit scheme
we need to assume that a one-sided Lipschitz condition on $f$ and monotone condition on
$L^{1}g$ hold only on the positive domain. This significantly relaxes the conditions
required for the existence and uniqueness of a solution to the implicit scheme
\eqref{eq:os-Milstein}.
\begin{assp}\label{as:one-sided}
Coefficients $f$ and $g$ of the equation \eqref{eq:SDE} are locally Lipschitz continuous and satisfy the
following two conditions: \newline
\textit{\underline{One-sided Lipschitz condition on $\RR_{+}$.}} 
There exists a constant $K>0$ such that
\begin{equation} \label{as:mono}
(x-y)(f(x)-f(y))\leq K\lev x-y \rev^{2}  \quad \hbox{for all} \quad x,y\in \RR_{+}.
\end{equation}
\textit{\underline{Monotone condition on $\RR_{+}$.}} Operator $L^{1}$ acting on $g$ satisfies 
\begin{equation} \label{as:diff}
(x-y)(L^{1}g(x)-L^{1}g(y) )\ge 0 \quad \hbox{for all} \quad  x,y\in\RR_{+}.
\end{equation}
\end{assp}
Many mean-reverting models with super- and sub-linear
diffusion coefficients  satisfy Assumption \ref{as:one-sided};
for example, the mean-reverting SDEs
\begin{equation*}
dx(t)= (\mu-x(t)^{q})dt + x(t)^{p}dw(t) \quad \hbox{for } x\in\RR,
\end{equation*}
with $\mu,q>0$ and $ p\ge 0.5$.

In general, boundary behavior of one-dimensional SDEs can be fully characterized
by the Feller test \cite{karatzas1991brownian}. 
Let us consider the interval $(0,\8)$. We assume that  $f$ and $g$ are locally Lipschitz continuous in $(0,\8)$
and that $g^{2}(x) > 0$, for $x\in (0,\8)$.
Let us also define the scale function
\[
 p(x) = \int_{c}^{x} \exp\left[  -2\int_{c}^{s}\frac{f(z)}{g^{2}(z)}dz \right]ds,
\]
where $c\in \RR$. Since we analyse the behaviour of the above function at $0$,
we assume that $c>x$. By Proposition 5.22 in \cite{karatzas1991brownian} we have that 
if $p(0+)=-\8$ then
$
 \PP \left[ \inf_{0\le t <\8} x(t) = 0   \right] = 1 
$.
Therefore, in order to show that the 
solution to \eqref{eq:SDE} is non-negative it is enough to 
show that $p(0+)=-\8$.
 \begin{assp}  \label{as:positivity}
The coefficients $f$ and $g$ in \eqref{eq:SDE} satisfy the following conditions: 
\begin{equation*}
f(0)\ge 0, \quad g^{2}(x) > 0 \; \hbox{for}\;  x\in (0,\8), \;  \hbox{and}\;  g(0)=0 \;\hbox{for} \; x=0.
\end{equation*}
 \end{assp}
To understand Assumption \ref{as:positivity} better, we proceed with a heuristic
argument. Lets  assume that the solution 
of \eqref{eq:SDE} attained $0$ at time $t$. Since the solution 
is Markovian we can consider the solution to this SDE with initial condition $x(t)=0$ that reads
\[
 dx(t) = f(0)dt + g(0)dw(t).
\]
It is clear that we need to have $g(0)=0$ and $f(0)\ge 0$ in order for $x(t)$ 
to stay non-negative. 
%
%
\begin{theorem} \label{th:Positivity}
Let Assumptions \ref{as:one-sided} and \ref{as:positivity}
hold. In addition we require that  
\[
L^{1}g(x) > 0 \quad  \hbox{for } x>0.
\]
Then there exists 
a 
unique positive 
solution to the $(\o,\s)$-Milstein scheme
\eqref{eq:os-Milstein} if
\begin{equation} \label{eq:timestep}
x - \frac{g^{2}(x)}{2L^{1}g(x)} + (1-\o) f(x)\D t-\frac{(1-\s)}{2}L^{1}g(x) \D t >  \,-\o f(0)\D t , \qquad  x > 0.
\end{equation}
Similarly,  
a 
unique
non-negative solution exists if 
\begin{equation} \label{eq:timestep2}
 x - \frac{g^{2}(x)}{2L^{1}g(x)} + (1-\o) f(x)\D t-
\frac{(1-\s)}{2}L^{1}g(x) \D t \ge  \, -\o f(0)\D t , \qquad  x > 0.
\end{equation}
\end{theorem}
\begin{proof}
In view of Lemma \ref{Zeidler} and Definition \ref{def:pos} in order to prove the lemma we analyse the following equation
\[
F(X_{k+1}) = X_{k}, \quad \forall k. 
\]  
First we prove that  $ \PP \{X_{k+1} > 0 \mid X_{k} >0 \}  = 1.$
By Assumptions \ref{as:one-sided} 
 operator $F$ in (\ref{eq:Fdef}) 
is monotone on $(0,\8)$ and we have 
\begin{equation*}
\lim_{ x \rightarrow\8}\frac{xF(x)}{\lev x \rev} = \8.  
\end{equation*}
By Assumption \ref{as:positivity} we arrive at
\[
 \lim_{ x \rightarrow 0^{+}}\frac{xF(x)}{\lev x \rev} = -\o f(0) \D t.
\]
Hence operator $F$ is coercive on $(0, \8)$.  
Due to Lemma \ref{Zeidler}, we may complete the proof by showing 
\begin{equation*}
b(x)=x+(1-\o) f(x)\D t +g(x)\D w_{t_{k+1}}+
\frac{1}{2}L^{1}g(x)\D w_{t_{k+1}}^{2}-\frac{(1-\s)}{2}L^{1}g(x) \D t > -\o f(0) \D t , \quad \hbox{for} \quad x>0,
\end{equation*}
from which it follows that there exists a positive solution to
$ F(X_{t_{k+1}})=b(X_{t_{k}}) $.
First, for any given $x>0$ we find the minimum of the function
\[
H(y)=g(x)y+\frac{1}{2}L^{1}g(x)y^{2}.
\]
Under the assumption $L^{1}g(x)>0$, for $x>0$, this function possesses a global minimum 
\[
 \min_{y}H(y) = - \frac{g^{2}(x)}{2L^{1}g(x)}.
\]
Hence
\[
b(x)\ge x+(1-\o) f(x)\D t-\frac{(1-\s)}{2}L^{1}g(x) \D t  - \frac{g^{2}(x)}{2L^{1}g(x)} > -\o f(0) \D t, \quad x>0,
\]
as required. 
For the non-negative case we have $b(x)\ge -\o f(0) \D t$, $x>0$. In that case we also need to check what happens if for some $k$ we have the following event $ \{X_{k+1}=0 \mid X_{k} >0 \}$
(that corresponds to the case where $b(x) = -\o f(0) \D t$  ). Then by Assumption \ref{as:positivity}, 
$ b(0) = (1-\o)f(0)\D t$ and we require that $b(0)\ge -\o f(0) \D t $. That holds due to Assumption \ref{as:positivity}.
\end{proof}
For the 
fully implicit $(1,1)$-Milstein scheme
we see from \eqref{eq:timestep}
 that a  condition 
guaranteeing  
non-negativity independently of $\D t$ is 
\[
 x - \frac{g^{2}(x)}{2L^{1}g(x)}  \ge 0, \qquad  x > 0.
\]
%
\subsection{Example: Heston Volatility Model}
Now we demonstrate that approximation of the
$3/2$-Heston volatility model
(\ref{eq:heston})
 with the double implicit
Milstein scheme preserves non-negativity. We point out
that implicitness in the numerical approximation does not increase
computational cost in this case, since we are able to find an explicit
solution. This often will be the case in mathematical finance, where typical models have 
drift and diffusion coefficients of a polynomial type. 

The $(1,1)$-Milstein scheme has the form
\begin{align} \label{eq:11heston}
X_{t_{k+1}}& = X_{t_{k}} + f(X_{t_{k+1}})\D t +
g(X_{t_{k}})\D w_{t_{k}} + \frac{1}{2}L^{1}g(X_{t_{k}})\D w_{t_{k}}^{2} - \frac{1}{2}L^{1}g(X_{t_{k+1}}) \D t,
\end{align}
where now 
$f(x)=\mu x-\a x^{2}$, 
$g(x)=\be x^{3/2}$ and 
$L^{1}g(x) = \frac{3}{2} \be^{2} x^{2}$.
Clearly, the coefficients of equation \eqref{eq:heston} satisfy Assumptions \ref{as:one-sided} and \ref{as:positivity}. Hence,
we may show 
that 
\eqref{eq:11heston} 
has 
a unique 
non-negative 
solution 
by 
verifying condition (\ref{eq:timestep2}) in 
Theorem \ref{th:Positivity}. This reduces to 
$ x - x/3  \ge 0$ for $ x \ge 0$ and the result follows.

An explicit formula for $X_{t_{k+1}}$ can be found by
solving the relevant quadratic equation and choosing the positive solution, to give  
\[
X_{t_{k+1}}=(2(\a+\frac{3}{4}\be^{2}\D t ))^{-1}\biggl(\sqrt{(1-\mu \D t)^2+4(\a+\frac{3}{4}\be^{2}) \D t
(X_{t_{k}} + \be X_{t_{k}}^{3/2}\D w_{t_{k}}+3/4\be^{2} X_{t_{k}}^{2}w^{2}_{t_{k}})   }-(1-\mu \D t) \biggr).
\]
\section{Stability Analysis} \label{sec:Milststability}
In this section we examine the global stability of the $(\s,\o)$-Milstein scheme
(\ref{eq:os-Milstein}).
The stability conditions we derive are related to mean-square 
stability, and we are interested
in results that do not put severe restrictions on the time step.
We begin with linear test equations where we can 
derive sharp results and represent stability
regions graphically. 
\subsection{Linear Mean-Square Stability}
For the linear test SDE
\begin{equation} \label{eq:testSDE}
dx(t)=\a x(t)dt+\mu x(t)dw(t),
\end{equation}
the property of 
mean-square stability,
\[
\lim_{t\rightarrow \8}\E\lev x(t)\rev^2=0,
\]
is characterized by 
\begin{equation}
(2\a+\mu^{2})<0.
\label{eq:mschar}
\end{equation}
For the $\o$-Milstein scheme
on 
(\ref{eq:testSDE}), 
\[
X_{t_{k+1}}=X_{t_{k}}+\o \a X_{t_{k+1}}\D t + (1-\o) \a X_{t_{k}}\D t +\mu X_{t_{k}}\D w_{t_{k+1}}
+\frac{1}{2}\mu^{2}X_{t_{k}}[\D w_{t_{k+1}}^{2}-\D t],
\]
the analogous property
\begin{equation}
\lim_{k\rightarrow \8}\E\lev X_{t_{k}}\rev^2=0,
\label{eq:milstab}
\end{equation}
was studied in \cite{higham2000stability}.
In particular,  
the linear stability region 
\begin{equation}
R_{MS}:=\{\D t \a,\D t \mu^{2}\in\RR :\hbox{method mean-square stable on (\ref{eq:testSDE})}  \}
\end{equation}
was shown to be significantly smaller than that for the 
corresponding Euler-based scheme. 
We now examine the new Milstein scheme (\ref{eq:os-Milstein})
in this setting,
which reduces to 
\begin{align} \label{osMilst}
X_{t_{k+1}}=X_{t_{k}}+\o \a X_{t_{k+1}}\D t + (1-\o) \a X_{t_{k}}\D t +\mu X_{t_{k}}\D w_{t_{k+1}}\\  \notag
+\frac{1}{2}\mu^{2}X_{t_{k}}\D w_{t_{k+1}}^{2}
-\frac{(1-\s)}{2}\mu^{2}X_{t_{k}} \D t-\frac{\s}{2}\mu^{2}X_{t_{k+1}} \D t.
\end{align}
 
\begin{theorem} The $(\o,\s)$-Milstein scheme 
\eqref{osMilst} is linearly mean-square stable, 
(\ref{eq:milstab}), 
if and only if
\begin{equation} \label{stabregion}
(2\a+\mu^{2})+\D t\a^{2}(1-2\o)+\frac{\D t \mu^{2}}{2}(2\s\a+\mu^{2})<0.
\end{equation}
\end{theorem}
\begin{proof}
We rewrite (\ref{osMilst})
as a recurrence of the form
\begin{equation*}
X_{t_{k+1}}=X_{t_{k}}\left( p+q\x_{t_{k+1}}+r\x^2_{t_{k+1}} \right),
\end{equation*}
where $\x \sim\NN(0,1)$ and 
\begin{equation*}
p=\frac{1+ (1-\o) \a \D t-\frac{(1-\s)}{2}\mu^{2}\D t}{1-\o \a \D t+\frac{\s}{2}\mu^{2}\D t},
\end{equation*}
\begin{equation*}
q=\frac{\mu\sqrt{\D t}}{1-\o \a \D t+\frac{\s}{2}\mu^{2}\D t},
\end{equation*}
\begin{equation*}
r=\frac{\frac{1}{2}\mu^{2}\D t}{1-\o \a \D t+\frac{\s}{2}\mu^{2}\D t}.
\end{equation*}
Then
\begin{equation*}
\lev X_{t_{k+1}}\rev^{2}=\lev X_{t_{k}}\rev^{2}\left( p^{2}+q^{2}\x_{t_{k+1}}^{2}+r^{2}\x^{4}_{t_{k+1}} +
2pq\x_{t_{k+1}}+2pr\x^{2}_{t_{k+1}}+2qr\x^{3}_{t_{k+1}}\right).
\end{equation*}
Taking conditional expectation of both sides lead us to
\begin{equation*}
\E[\lev X_{t_{k+1}}\rev^{2}|\F_{t_{k}}]=\lev X_{t_{k}}\rev\left( p^{2}+q^{2}+3r^{2}+2pr\right).
\end{equation*}
Taking conditional expectation of both sides again we obtain
\begin{equation} \label{eq:ms}
\E\lev X_{t_{k+1}}\rev^{2}=\E\lev X_{t_{k}}\rev^{2}\left( p^{2}+q^{2}+3r^{2}+2pr\right).
\end{equation}
Therefore 
stability 
is
characterized by 
$
(p+r)^{2}+q^2+2r^2<1
$.
This is equivalent to \eqref{stabregion}, as required.
\end{proof}
\begin{rmk} \label{rmk:milst}
Let us observe that for $\o =0.5$ and $\s=1$ we
have recovered precisely the condition
(\ref{eq:mschar})
for the underlying SDE, 
so the method perfectly reproduces stability for any step-size.
More generally, for $\o \ge 0.5$ and $\s = 1$ we have the property that 
``problem stable implies method stable for all $\D t$'', which is refered to as
A-stability in the deterministic literature. 
\end{rmk}
Motivated by \cite{higham2000stability,higham2001mean} we will draw stability regions
for \eqref{osMilst} in the $x$-$y$ plane, where $x=\a \D t$ and $y=\mu^2 \D t$.
In Figure \ref{fig:stability} the stability region of the
underyling SDE \eqref{eq:testSDE} is shaded light grey.
The upper pictures in Figure \ref{fig:stability} superimpose the
stability region of the $(\o,0)$-Milstein scheme with 
$\o=0,0.5,1$, respectively, using darker shading.
We see that even in the case of a linear scalar equation we are not able to
reproduce the stability region of the underyling test equation \eqref{eq:testSDE}.
However, by introducing additional implicitness we overcome this poor
performance. The lower pictures in Figure~\ref{fig:stability}
superimpose the stability region of the $(\o,\s)$-Milstein scheme
with $(0,1),(0.5,1),(1,1)$, respectively. As stated in Remark \ref{rmk:milst},
we recover exactly the stability region of underlying test SDE \eqref{eq:testSDE}
for $\o=0.5$ and $\s=1$. 
\begin{figure}[htb]
\begin{center}
 \includegraphics[scale=0.7]{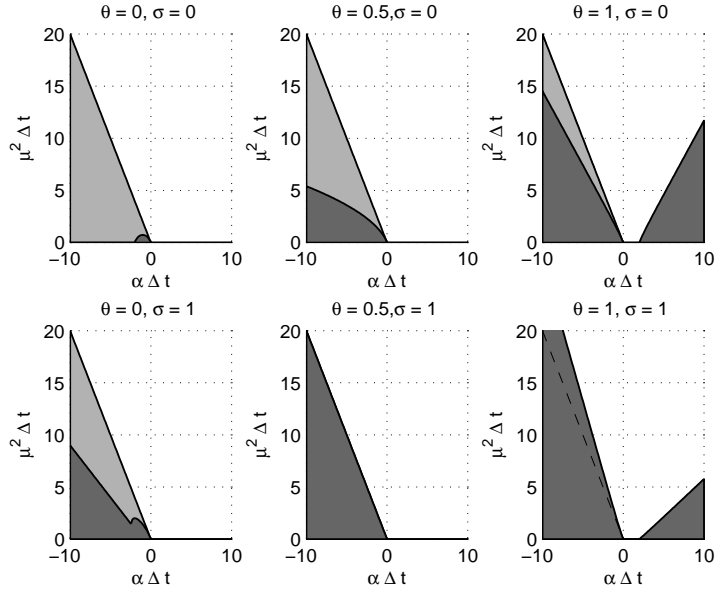}
\end{center}
        \caption{Light shading: 
linear mean-square stability of the SDE.
Darker shading: linear mean-square stability of 
Implicit Milstein 
    (upper) and  double-implicit Milstein (lower).}
\label{fig:stability}
\end{figure}
\subsection{Lyapunov Stability}
We begin this section by stating a result that 
combines
Doob's Decomposition and Martingale Convergence Theorems.
\begin{theorem}
[Lipster and Shiryaev \cite{liptser1989theory}]
 \label{L2}
Let $Z=\{Z_{n}\}_{n\in\NN}$ be a nonnegative decomposable stochastic
process with Doob-Meyer decomposition
${Z_{n}}=Z_{0}+A_{n}^{1}-A_{n}^{2}+M_{n}$, where
$A^{1}=\{A_{n}^{1}\}_{n\in\NN}$ and $A^{1}=\{A_{n}^{1}\}_{n\in\NN}$
are a.s. nondecreasing, predictable processes with
$A_{0}^{1}=A_{0}^{2}=0$, and $M=\{M_{n}\}_{n\in\NN}$ is local
$\{ \F_{n} \}_{n\in\NN}$-martingale with $M_{0}=0$.
Then
\begin{equation*}
\left\{ \omega :\lim_{n \rightarrow  \8}A^{1}(n)<\infty \right\} \subseteq \left\{
\omega :\lim_{n \rightarrow  \8}A^{2}(n)<\infty \right\} \cap \left\{
\lim_{n \rightarrow  \8}Z_{n}<\8  \right\} \quad \hbox{a.s.}
\end{equation*}
\end{theorem}
In \cite{shen2006improved}, the authors proved a very general Stochastic LaSalle Theorem.
Here we present a simplified version of their theorem,
with fixed Lyapunov function $V(x)=\lev x\rev^{2}$.
\begin{theorem}[Shen et al. \cite{shen2006improved}] \label{T2}
Let local Lipschitz conditions hold for $f$ and $g$.
Assume further that there exists a function $z \in C(\RR;\RR_{+})$ such that
\begin{equation}
x f(x) +\frac{1}{2}\lev g(x) \rev^{2}\le -z(x),
 \label{eq:T2cond}
\end{equation}
for all $x\in \RR$.
  For any $x_{0}\in \RR$, the solution $(x(t))_{t\ge0}$ of \eqref{eq:SDE} then has the properties that
\[
\limsup_{t \rightarrow \8}\lev x(t)\rev^{2}<\8
 \qquad \mathrm{a.s.} 
\qquad \mathrm{and}
\qquad
\lim_{t \rightarrow \*} z(x(t)) = 0 \qquad \mathrm{a.s.}
\]
Further if $z(x)=0$ if and only if $x=0$, then
\[
\lim_{t \rightarrow \8}x(t)=0 \qquad \mathrm{a.s.} 
\quad \hbox{$\forall$ $x\in\RR$}.
\]
\end{theorem}
Now we present a counterpart of this Stochastic LaSalle Theorem for
the new Milstein scheme. 
\begin{theorem} \label{th:Stabilty}
Let Assumption \ref{as:one} hold.
Assume that for the
$(\o,\s)$-Milstein Scheme \eqref{eq:os-Milstein} there exists a function  $z \in C(\RR;\RR_{+})$ such that
\begin{align} \label{eq:stab}
&2 x f(x) + \lev g(x)\rev^{2}+(1-2\o)\lev f(x)\rev^{2} \D t  \nonumber  \\
&+\frac{\D t}{2} L^{1}g(x) (2 \s f(x)+ L^{1}g(x) )
\leq -z(x) \quad \quad \hbox{for all } x \in\RR. 
\end{align}
Then
\[
\limsup_{k\rightarrow \infty }\left\vert X_{t_{k}}\right\vert ^{2}<\8 
\]
and
\[
\lim_{k\rightarrow \infty } z(X_{t_{k}}) =0 \qquad \mathrm{a.s.}  
\]
Further if $z(x)=0$ if and only if $x=0$ 
then
\begin{equation*}
\lim_{k\rightarrow \infty }  X_{t_{k}} =0 \qquad \mathrm{a.s.}
\end{equation*}
\end{theorem}
\begin{proof}
We can rewrite $F$ in (\ref{eq:Fdef})
as
\begin{align*}
F(X_{t_{k+1}})
& = F(X_{t_{k}}) + f(X_{t_k})\D t + g(X_{t_k})\D w_{t_{k+1}}
 + \frac{1}{2}L^{1}g(X_{t_{k}}) ( \D w_{t_{k+1}}^{2} - \D t ).
\end{align*}
Squaring both sides, we arrive at
\begin{align*}
\lev F(X_{t_{k+1}}) \rev^{2}  = &
\lev F(X_{t_{k}}) \rev^{2}  + \lev f(X_{t_{k}})\D t \rev^{2} + \lev g(X_{t_{k}})\rev^{2}\D t
+ \frac{1}{2} \lev L^{1}g(X_{t_{k}}) \rev^2 \D t^2 
 + 2 F(X_{t_{k}}) f(X_{t_{k}})\D t + m_{k+1},
\end{align*}
where
\begin{equation} \label{eq:martingle}
\begin{split}
m_{k+1} =  & \lev g(X_{t_{k}}) \rev^2 ( \D w_{t_{k+1}}^{2} - \D t )
   + \frac{1}{2} \lev L^{1}g(X_{t_{k}}) \rev^2 [( \D w_{t_{k+1}}^{2} - \D t )^2 - 2\D t^2 ] \\
& + 2 F(X_{t_{k}})\left[  g(X_{t_{k}})\D w_{t_{k+1}} + \frac{1}{2}L^{1}g(X_{t_{k}}) ( \D w_{t_{k+1}}^{2} - \D t ) \right]\\
&+ 2 f(X_{t_{k}})\D t \left[ g(X_{t_{k}})\D w_{t_{k+1}}
 + \frac{1}{2}L^{1}g(X_{t_{k}}) ( \D w_{t_{k+1}}^{2} - \D t ) \right]  \\
& +  g(X_{t_{k}}) L^{1}g(X_{t_{k}}) ( \D w_{t_{k+1}}^{2} - \D t ) \D w_{t_{k+1}}
\end{split}
\end{equation}
is a local martingale difference. From the definition of $F$ 
we arrive at
\begin{align} \label{eq:F-ine}
\lev F(X_{t_{k+1}}) \rev^{2}  = &
\lev F(X_{t_{k}}) \rev^{2}  + 2 X_{k} f (X_{t_{k}})\D t  + \lev g(X_{t_{k}})\rev^{2}\D t
+ \frac{1}{2}  L^{1}g(X_{t_{k}}) \left[ L^{1}g(X_{t_{k}})  + 2 \s  f (X_{t_{k}})\right] \D t^{2} \nonumber \\
& +    ( 1- 2 \o) \lev f(X_{t_{k}}) \rev^{2} \D t^{2}  + m_{k+1}. 
\end{align}
Therefore
\[
\lev F(X_{t_{k+1}})\rev^2=\lev F(X_{t_{k}})\rev^2 - 
        A_{t_{k}}\D t + m_{k+1},
\]
where
\begin{eqnarray*}
A_{t_{k}}(x) & = & -\left(2 X_{t_{k}} f (X_{t_{k}}) + \lev g(X_{t_{k}})\rev^{2} + \frac{1}{2} 
L^{1}g(X_{t_{k}}) \left[ L^{1}g(X_{t_{k}})  + 2 \s 
 f (X_{t_{k}})\right] \D t \right.\\
&& \mbox{} + \left.  ( 1- 2 \o) \lev f(X_{t_{k}}) \rev^{2} \D t 
       \right). 
\end{eqnarray*}
Hence, we have obtained a decomposition that allows 
us to apply Theorem \ref{L2}, i.e.,
\[
\lev F(X_{t_{N+1}})\rev^2
=\lev F(X_{t_{0}})\rev^2-\sum_{k=0}^{N}A_{t_{k}}\D t
+\sum_{k=0}^{N}m_{k+1}.
\]
Theorem~\ref{L2} gives
$
\lim_{k\rightarrow \infty }\left\vert 
 F(X_{t_{k}})\right\vert ^{2}<\infty 
$.
By condition \eqref{eq:stab} and \eqref{eq:lowerbound}
\begin{align*} 
 \lev F(x) \rev^{2} = & \left(  x-\o f(x)\D t+\frac{1}{2}\s L^{1}g(x)\D t    \right)^{2}  \nonumber\\
  = & \lev x \rev^{2} - 2 \o x f(x) \D t - \o \s f(x) L^{1}g(x)\D t^{2} 
+ \o^{2}(f(x))^{2}\D t^{2} + \frac{1}{4} \s^{2}(L^{1}g(x))^{2}\D t^{2} + \s x L^{1}g(x) \D t  \nonumber \\
\ge &  \lev x \rev^{2} + \o z(x)\D t + \s x L^{1}g(x) \D t  \\
\ge & \lev x \rev^{2}  - \s \lev x \rev \lev L^{1}g(0) \rev \D t \\
\ge & (1-0.5 \D t) \lev x \rev^{2}  - 0.5 \,\s^{2}  \lev L^{1}g(0) \rev^{2} \D t.
 \end{align*} 
Hence $\limsup_{k\rightarrow \infty }\left\vert X(t_{k})\right\vert ^{2}$
exists and is finite almost surely.
Another implication of Theorem \ref{L2} is
\begin{align*}
\sum_{k=0}^{\8} z(X_{t_{k}}) \D t\le\sum_{k=0}^{\8}A_{t_{k}}\D t <\8 \quad\text{a.s},
\end{align*}
as required.
\end{proof}
In the case where  \eqref{eq:os-Milstein}  is non-negative it is enough if condition 
\eqref{eq:stab} holds on the non-negative half line. 
\begin{theorem}
Let conditions required for existence of non-negative solution in Theorem \ref{th:Positivity} hold.
Assume that for the
$(\o,\s)$-Milstein Scheme \eqref{eq:os-Milstein} there exists a function  $z \in C(\RR;\RR_{+})$ such that
\begin{align*} 
&2 x f(x) + \lev g(x)\rev^{2}+(1-2\o)\lev f(x)\rev^{2} \D t  \nonumber  \\
&+\frac{\D t}{2} L^{1}g(x) (2 \s f(x)+ L^{1}g(x) )
\leq -z(x) \quad \quad \hbox{for all } x \in\RR_{+}.
\end{align*}
Then
\[
\limsup_{k\rightarrow \infty }\left\vert X(t_{k})\right\vert ^{2}<\8 
\]
and
\[
\lim_{k\rightarrow \infty } z(X_{t_{k}}) =0 \qquad \mathrm{a.s.}  
\]
Further if $z(x)=0$ if and only if $x=0$ 
then
\begin{equation*}
\lim_{k\rightarrow \infty }  X_{t_{k}} =0 \qquad \mathrm{a.s.}
\end{equation*}
\end{theorem}
\begin{proof}
The proof is analogous to the proof of Theorem \ref{th:Stabilty}. 
\end{proof}

\begin{rmk} \label{rmk:milstb}
Following on from Remark~\ref{rmk:milst}, suppose
that 
(\ref{eq:T2cond}) holds, so the results of Theorem~\ref{T2}
hold for the SDE. Then, to minimize restrictions on the 
stepsize in 
(\ref{eq:stab}), the choice $\theta = 0.5$ 
is clearly best, and the extra freedom allowed by the parameter
$\sigma$ can be used to exploit dissipativity. For example,
on the SDE
\[
   dx(t) = -x(t)^3 dt + x(t)^2 dw(t),
\]
we have 
\[
L^{1}g(x) (2 \s f(x)+ L^{1}g(x) ) = 
L^{1}g(x) (- 2 \s x^3 + 2 x^3 ),
\]
so the choice $\s = 1$ makes 
(\ref{eq:stab}) 
independent of $\D t$ 
and identical to 
(\ref{eq:T2cond}).
\end{rmk}
\section{Convergence Result}
In this section we show that the numerical approximation \eqref{eq:os-Milstein} strongly converges 
to the solution of \eqref{eq:SDE} under fairly general conditions. We will not establish 
the rate of convergence, but we perform numerical experiments that 
suggest a rate of $1$.
We note that the $(1,1)$ scheme was considered in 
\cite{kloeden1992numerical} as an alternative to 
the more typical $(1,0)$ version. In particular, those authors showed that when 
the coefficients $f$, $g$ and $L^{1}g(x)$ in \eqref{eq:SDE} are globally 
Lipschitz, the $(1,1)$ case retains the usual first order of strong convergence.
This result is easily extended to the general $(\o,\s)$ case.
\begin{theorem} \label{th:strongkp}
Let $f$, $g$ and $L^{1}g(x)$  be globally 
Lipschitz. Then the $(\o,\s)$-Milstein scheme 
\eqref{eq:os-Milstein} strongly converges to the solution 
of the SDE \eqref{eq:SDE}, that is 
\[
  \E \left[ \sup_{0\le t_{k}\le T}\lev x(t_{k}) - X_{t_{k}}\rev^{p}\right] = \D t ^{p} \quad \hbox{for }  p\ge 2.
\]
\end{theorem}
\begin{proof}
 A proof follows by extending the $(1,1)$ case from Chapter~12 of Kloeden and Platen \cite{kloeden1992numerical}.
\end{proof}
Then using a localization procedure as in \cite{gyongy1998note,jentzen2009pathwise} we can 
prove pathwise convergence without global Lipschitz Assumption. From \cite{jentzen2009pathwise} 
we know that scheme \eqref{eq:os-Milstein} almost surely 
converges to the solution of \eqref{eq:SDE}, that is: 
\begin{theorem}
Assume that the solution to SDE \eqref{eq:SDE} has a strong solution. Then  
the $(\o,\s)$-Milstein scheme \eqref{eq:os-Milstein} converges to the solution 
of the SDE \eqref{eq:SDE} in the pathwise sense, that is
for $\g > 0$ there exists a random variable $K=K(\w)$, $\w \in \Omega$,
 such that 
\begin{equation} \label{eq:asconv}
 \sup_{0\le t_{k}\le T} \lev x(t_{k}) - X_{t_{k}} \rev \le K(\w) \D t^{1- \g}, \quad
\hbox{for  } T,\g>0.
\end{equation}
\end{theorem}
\begin{proof}
 A proof can be written in an analogous way 
  to the proof of Theorem~1 in \cite{jentzen2009pathwise},  
or Theorem~2.3 in \cite{gyongy1998note}. 
A key difference is that those 
authors used explicit schemes and 
defined continuous extensions 
to overcome some technical difficulties. 
In our setting, we need to define an
$\F_{t}$-adapted continuous extension of the approximation \eqref{eq:os-Milstein}.
Using notation of Lemma 2.4 we can write  
\eqref{eq:os-Milstein} in the form
\begin{align*} 
X_{t_{k+1}}& = F^{-1} \Biggl( X_{t_{k}}
                 + (1-\o) f(X_{t_{k}})\D t +
                 g(X_{t_{k}})\D w_{t_{k}}
                  + \frac{1}{2}L^{1}g(X_{t_{k}})\D w_{t_{k}}^{2} 
             - \frac{(1-\s)}{2}L^{1}g(X_{t_{k}}) \D t \Biggr).               
 \end{align*}
Then a suitable continuous extension of \eqref{eq:os-Milstein} for $t \in[t_{k}, t_{k+1})   $ could be defined by 
\begin{align*} 
X(t)& = F^{-1} \Biggl( X_{t_{k}}
                 + (1-\o) f(X_{t_{k}})(t-t_{k}) +
                 g(X_{t_{k}}) ( w(t) - w(t_{k} ) \\
                & + \frac{1}{2}L^{1}g(X_{t_{k}}) ( w(t) - w(t_{k} )^{2}    
             - \frac{(1-\s)}{2}L^{1}g(X_{t_{k}}) (t-t_{k})\Biggr).               
 \end{align*}
\end{proof}
In order to show that we also have strong convergence we need to show 
that the solution to \eqref{eq:os-Milstein} has bounded moments. 
We will prove boundedness of the moments under the 
following assumption. 
\begin{assp} \label{a0}
\textit{\underline{Monotone-type condition}.} There exist constants $a$ and $b$ such that
\begin{align} \label{eq:moments}
&2 x f(x) + \lev g(x)\rev^{2}+(1-2\o)\lev f(x)\rev^{2} \D t  \nonumber  \\
&+\frac{\D t}{2} L^{1}g(x) (2 \s f(x)+ L^{1}g(x) )
\leq a + b \lev x \rev^{2} \quad \quad \hbox{for all } x\in\RR. 
\end{align}
\end{assp}
The 
following lemma establishes a useful relation between function $F(x)$ defined in \ref{eq:Fdef} and its argument $x$. 
\begin{lemma} \label{lem:FF}
 Lets Assumptions \ref{as:one} and \ref{a0} hold. Then there exist constants $c_{1},c_{2}>0$ such that
\[
 \lev F(x) \rev^2 \ge c_{1}\lev x\rev^{2} - c_{2}\D t \quad \hbox{for } x\in\RR.
\]
\end{lemma}
\begin{proof}
 By Assumptions \ref{as:one} and \ref{a0} we have 
\begin{equation} \label{eq:estiamtes}
\begin{split}
 \lev F(x) \rev^{2}  &
\ge   \lev x \rev^{2} - \o a \D t - \o b \lev x \rev^{2} \D t + \s x L^{1}g(x) \D t  \\
& \ge \lev x \rev^{2} - \o a \D t - \o b \lev x \rev^{2} \D t - \s  \lev x \rev \lev L^{1}g(0)\rev \D t\\
& \ge \lev x \rev^{2} - \o a \D t - \o b \lev x \rev^{2} \D t -   \o/2\lev x \rev^2 \D t - \s^{2}/(2\o)\lev L^{1}g(0)\rev^{2} \D t \\
&\ge  (1-(\o b + \o/2) \D t )\lev x \rev^{2} - \o a \D t - \s^{2}/(2\o)\lev L^{1}g(0)\rev^{2} \D t,
\end{split}
 \end{equation}
and we take $c_{1} = (1-(\o b + \o/2) \D t )$ and $c_{2}= - \o a  + \s^{2}/(2\o)\lev L^{1}g(0)\rev^{2}$.
Due to \eqref{eq:time_re} $c_{1}>0$.
\end{proof}
Our analysis uses a 
 localization procedure.
We define the stopping time $\l_{m}$ by
\begin{equation} \label{sd}
\lambda _{m}=\inf \{k: \lev X_{t_{k}}\rev > m \}.
\end{equation}
We observe that when $k \in[0,\l_{m}(\w)]$,
$\lev X_{t_{k-1}}(\w) \rev \le m$,
but we might have that $\lev X_{t_{k}} (\w)   \rev > m$,
so the following lemma is not trivial.
\begin{lemma} \label{stopping}
Let Assumptions \ref{as:one} and \ref{a0} hold. Then for $p\geq 2$ 
and sufficiently large integer $m$, there exists a
constant $K=K(p,m)$, such that
\begin{equation*}
\E\left[\lev X_{t_{k}}\rev ^{p}\1_{[0,\lambda
_{m}]}(k)\right]<K \quad \mathrm{~~for~any~~}k\ge0.
\end{equation*}
\end{lemma}
\begin{proof}
By \eqref{eq:F-ine} and Assumption \ref{a0} we obtain 
\begin{align}
\lev F(X_{t_{k}})\rev^2
 \le & \lev F(X(t_{k-1})\rev^2 +  a \D t + b \lev X{t_{k-1} } \rev^{2}  \D t +  \D m_{k} , \nonumber
\end{align}
where $\D m_{k+1}$ is defined by \eqref{eq:martingle}.
Using the basic inequality $(a_{1} + a_{2} + a_{3}+ a_{4})^{p/2}\le 4^{p/2-1} ( a_{1}^{p} + a_{2}^{p} + a_{3}^{p} + a_{4}^{p}   ) $,
where $a_{i}\ge0$, we obtain
\begin{align}
\lev F(X_{t_{k}})\rev^p
 \le & 4^{p-1} \left( \lev F(X_{t_{k-1}}\rev^p +  (a \D t)^{p/2}
 + b \lev X_{t_{k-1} } \rev^{p}  \D t + \lev \D m_{k} \rev^{p/2} \right).
\end{align}
As a consequence
\begin{align*}
\E \left[ \lev F(X_{t_{k}})\rev^p \1_{[0,\l_{m}]}(k) \right]
 \le & 4^{p-1} \biggl( \E \left[ \lev F(X_{t_{k-1}}\rev^p \1_{[0,\l_{m}]}(k) \right]  +  (a \D t)^{p/2} \\
 & + b  m^{p}  \D t + \E \left[ \lev \D m_{k} \rev^{p/2} \1_{[0,\l_{m}]}(k) \right] \biggr).
\end{align*}
In order to bound $\E \left[ \lev \D m_{k} \rev^{p/2} \1_{[0,\l_{m}]}(k) \right]$ we need to consider 
all the terms of $\D m_{k} $ separately. 
By the Cauchy-Schwarz inequality  
\begin{align*}
& \E \biggl[   \lev g(X_{t_{k-1}}) \rev^{p} \lev \D w_{t_{k+1}}^{2} - \D t \rev^{p/2}  \biggr] \1_{[0,\l_{m}]}(k) \\
& \le   \biggl[   \bigl(\E\bigl[\lev g(X_{t_{k-1}}) \rev^{2p}\1_{[0,\l_{m}]}(k)\bigr]\bigr)^{1/2}
	\bigl(\E\lev \D w_{t_{k+1}}^{2} - \D t \rev^{p}\bigr)^{1/2}
           \biggr].
\end{align*}
Since there exists a positive constant $C(p)$,
such that $\E\lev\Delta w_{t_{k-1}}\rev^{2p}<C(p)$, there exists a constant $C(m,p)$ such that 
\begin{align*}
& \E \biggl[   \lev g(X_{t_{k-1}}) \rev^{p} \lev \D w_{t_{k+1}}^{2} - \D t \rev^{p/2}  \biggr] \1_{[0,\l_{m}]}(k) 
\le C(m,p).
\end{align*}
In the same way we can bound all the other terms of $\D m_{k}$. 
Hence
\begin{align*}
\E\left[\lev F (X_{t_{k}} ) \rev ^{p}\1_{[0,\l_{m}]}(k)\right]<C(m,p).
\end{align*}
Due to Lemma \ref{lem:FF} the proof is complete.
\end{proof}
In addition to Assumption \ref{a0} we require the following very mild restriction on the coefficients 
of the SDE. 
\begin{assp}\label{as:polynomial}
The coefficients of equation (\ref{eq:SDE})
satisfy a polynomial growth condition. That is,  there exists a pair of constants  $h\ge 1$ and $H>0$  such that
\begin{equation} \label{ass:P}
\lev f(x) \rev \vee \lev g(x)\rev \leq H ( 1 + \lev x \rev ^{h} ), 
\qquad \forall x.
\end{equation}
\end{assp}
Now we formulate the key theorem that allows us 
to prove a strong convergence result. 
\begin{theorem} \label{th:bound}
Let Assumptions \ref{as:one}, \ref{a0} and \ref{as:polynomial} hold.
Then there exists a constant $K=K(T)$ such that the 
$(\o,\s)$-Milstein scheme \eqref{eq:os-Milstein} satisfies 
\[
\sup_{0\le t_{k} \le T }\E \lev X_{t_{k}} \rev^{2} \le K.
\]
\end{theorem}
\begin{proof}
By \eqref{eq:F-ine} and Assumption \ref{a0} we arrive at
\begin{align} \label{eq:ineq}
\lev F(X_{t_{k+1}})\rev^2
 \le & \, \lev F(X_{t_{k}})\rev^2 +  a \D t + b \lev X_{t_{k} } \rev^{2}  \D t + \D m_{k+1},
\end{align}
where $\D m_{k+1}$ is defined by \eqref{eq:martingle}.
Let $N$ be any non-negative integer such that $N \D t \le T$.  
Summing both sides of
inequality \eqref{eq:ineq} from $k=0$ to $N\we\l_m$, we get
\begin{equation} \label{f1}
\begin{split}
\lev F( X_{t_{N\we\l_m+1}} ) \rev ^{2}
&\leq
\lev F( X_{t_{0}} )\rev ^{2} +a T
  + b \sum_{k=0}^{N\we\l_m} \lev X_{t_{k}} \rev^{2} \D t
       +\sum_{k=0}^{N\we\l_m} \D m_{k+1} \notag \\
     & \le
    \lev F( X_{t_{0}} )\rev ^{2} +a T
  + b \sum_{k=0}^{N} \lev X_{t_{k\we \l_{m}}} \rev^{2} \D t
       +\sum_{k=0}^{N} \D m_{k+1} \1_{[0,\lambda _{m}]}(k).
\end{split}
\end{equation}
Due to Lemma \ref{stopping} 
$ \sum_{k=0}^{N} \D m_{k+1} \1_{[0,\lambda _{m}]}(k)$ is a martingale. Hence 
\begin{align*}
  \E
 \lev F( X_{t_{N\we\l_m+1}} ) \rev^{2} \notag
   &\le
 \lev F( X_{t_{0}} ) \rev +a T + b \, \E
   \left[
   \sum_{k=1}^{N} \lev X_{t_{_{k\we \l_{m}}}} \rev^{2} \D t
   \right].
\end{align*}
Due to Lemma \ref{lem:FF} we have 
\begin{align*}
  \E
 \lev F( X_{t_{N\we\l_m+1}} ) \rev^{2} \notag
   &\le
 \lev F( X_{t_{0}} ) \rev +(a +c_{2}\,c_{1}^{-1})\,T   + b\,c_{1}^{-1} \E
   \left[
   \sum_{k=0}^{N} \lev F( X_{t_{_{k\we \l_{m}}}} ) \rev^{2} \D t
   \right].
\end{align*}
By the discrete Gronwall Lemma 
\begin{equation} \label{f1b}
 \E
    \lev F( X_{t_{N\we\l_m+1}} ) \rev^{2}
 \le
   \left[ \lev F( X_{t_{0}} ) \rev +(a +c_{2}\,c_{1}^{-1})\,T \right]
 \exp\left(b\,c_{1}^{-1}\,T\right),
  \end{equation}
where we used the fact that $N\D t\le T$. Thus, letting $m\rightarrow \infty $ in (\ref{f1b}) and applying
Fatou's lemma, we obtain 
\[
\ \E
    \lev F (X_{t_{N+1}} )\rev^{2}
 \le
 \left[ \lev F( X_{t_{0}} ) \rev +(a +c_{2}\,c_{1}^{-1})\,T \right]
 \exp\left(b\,c_{1}^{-1}\,T\right).
  \]
The final bound follows from Lemma \ref{lem:FF}.
\end{proof}
We are ready to prove a strong convergence result.
\begin{theorem} \label{th:strong}
Let Assumptions \ref{as:one}, \ref{a0} and \ref{as:polynomial} hold.
Then the $(\o,\s)$-Milstein scheme 
\eqref{eq:os-Milstein} strongly converges to the solution 
of the SDE \eqref{eq:SDE}, that is 
\begin{equation} \label{eq:Lpconv}
 \lim_{\D t \rightarrow 0 } \E \lev x(t_{k}) - X_{t_{k}}\rev^{p} =0 \quad \hbox{for } 0< p<2.
\end{equation}
\end{theorem}
\begin{proof}
By \eqref{eq:asconv} the $(\o,\s)$-Milstein approximnation
\eqref{eq:os-Milstein}  $X_{t_{k}}$ converges to $x(t_{k})$ in probability (Theorem 2.2 in \cite{shiryaev1996probability}).
Theorem \ref{th:bound} implies that the sequence  $\{ \lev X_{t_{k}} \rev^{2-\e}\}_{t_{k}}$ is uniformly integrable 
(Lemma 2.3 in \cite{shiryaev1996probability}). 
Therefore by the Vitali convergence theorem  (Theorem 2.4 in \cite{shiryaev1996probability} )
the statement of the theorem  holds.
\end{proof}
In case where we can guarantee non-negativity of approximation, conditions required to prove Theorem
\ref{th:strong} can be significantly relaxed.  

\begin{assp} \label{a0plus}
\textit{\underline{Monotone-type condition on $\RR_{+}$}.} There exist constants $a$ and $b$ such that
\begin{align} 
&2 x f(x) + \lev g(x)\rev^{2}+(1-2\o)\lev f(x)\rev^{2} \D t  \nonumber  \\
&+\frac{\D t}{2} L^{1}g(x) (2 \s f(x)+ L^{1}g(x) )
\leq a + b \lev x \rev^{2} \quad \quad \hbox{for all } x\in\RR_{+}. 
\end{align}
\end{assp}

\begin{theorem} 
Let conditions required for existence of non-negative solution in Theorem \ref{th:Positivity} hold.
Then under Assumptions \ref{as:polynomial} and \ref{a0plus} the $(\o,\s)$-Milstein scheme 
\eqref{eq:os-Milstein} strongly converges to the solution 
of the SDE \eqref{eq:SDE}, that is 
\begin{equation} 
 \lim_{\D t \rightarrow 0 } \E \lev x(t_{k}) - X_{t_{k}}\rev^{p} =0 \quad \hbox{for } 0< p<2.
\end{equation}
\end{theorem}
\begin{proof}
The theorem can be proved in an analogous way to Theorem~\ref{th:strong}
\end{proof}

It is clear that 3/2-model \eqref{eq:heston} doest not satisfy Assumption \ref{a0}, but satisfies
Assumption \ref{a0plus} as long as $\a\ge \be^{2}/4$. These condition seems not be restrictive
as pointed out in \cite{goard2011stochastic}. 

\subsection{Numerical Experiment}
In order to estimate the rate of convergence we proceed with numerical experiments
for
(\ref{eq:heston}).
We focus on
the strong endpoint error,
$
e^{\mathrm{strong}}_{\D t}=\E \lev x(T)-X_{T}\rev
$, 
with $T=1$.
We used $\mu = 0.1$, $\alpha = 0.2$, $\beta = \sqrt{0.2}$ and $x(0) = 0.5$. 
We plot $e^{\mathrm{strong}}_{\D t}$ 
against $\D t$ on a log-log scale.
Error bars
representing $95\%$ confidence intervals are shown by circles, and a reference
line of slope $1$ is also given.
\begin{figure}[htb]
\begin{center}
 \includegraphics[scale=0.8]{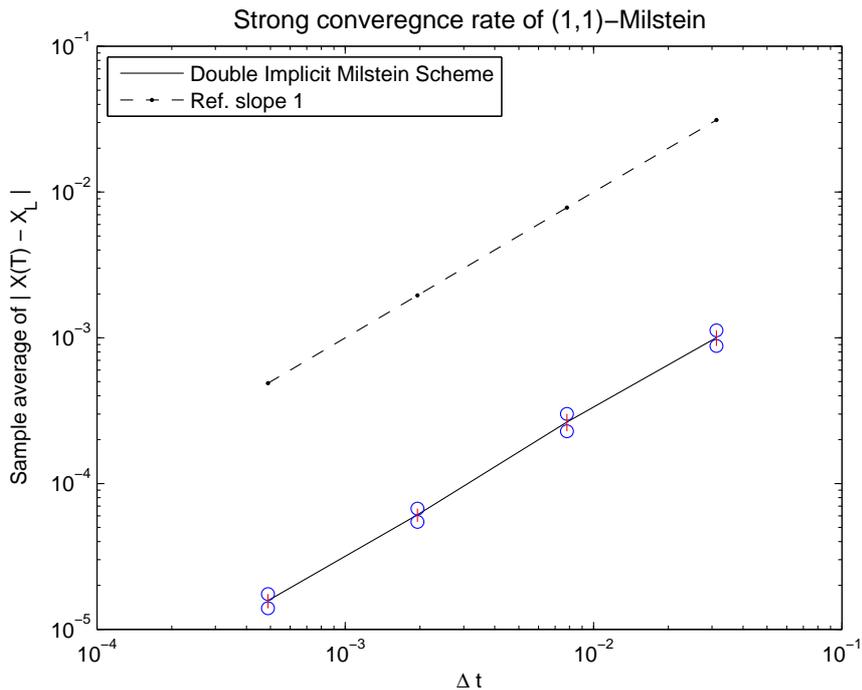}
\end{center}
        \caption{ Strong error of double-implicit Milstein scheme applied to
        Heston 3/2 Stochastic volatility model.
}
\label{fig:strong}
\end{figure}
Although we do not know the explicit form of the solution, Theorem \ref{th:strong} guarantees
that the (1,1)-Milstein scheme \eqref{eq:11heston} strongly converges to the true solution.
We therefore take the (1,1)-Milstein scheme
with $\D t = 2^{-14}$
as a reference solution. We compare this with the (1,1)-Milstein scheme evaluated with
$( 2\D t, 2^{3}\D t, 2^{5}\D t, 2^{7}\D t)$
in order to estimate the rate of convergence. Since we are using a Monte Carlo method,
the sampling error decays like $1/\sqrt{M}$, where $M= 10000$
is the number of sample paths. 
From Figure~\ref{fig:strong} we see that there appears to exist a positive constant $C$
such that
\begin{equation*}
e^{\mathrm{strong}}_{\D t}\le C \D t,  \quad \hbox{for sufficiently small $\D t$}.
\end{equation*}
A least squares fit for 
equality produced the value $1.1304$ for the rate
with residual of $0.2468$.  Hence, our results are consistent
with strong order of convergence equal to one.
\section{Conclusions}
Our aim was to introduce a new discretization scheme that 
can be shown to work well on highly nonlinear SDEs
arising in mathematical finance and to possess excellent 
linear and nonlinear 
stability properties.
There are several interesting areas for follow-up work; most notably
(a)
establishing a strong order of convergence for this method in 
a nonlinear setting, and 
(b) 
developing a theory of positivity preservation in the case of 
SDE systems and their numerical simulation.

\bibliographystyle{plain}
\bibliography{upthesis}

\end{document}